%% file: gromov-metrik.tex
\documentclass{article}

\usepackage{wollatex}
\usepackage{format}
\input macros-gromov.tex

\author{Wolfgang L\"ohr\footnote{University of Duisburg-Essen, Mathematics department, Universit\"atsstr.~2,
45141 Essen, Germany\newline\hspace*{1.8em}Email: wolfgang.loehr@uni-due.de}}
\title{Equivalence of Gromov-Prohorov- and Gromov's $\dg$-Metric on the Space of Metric Measure
Spaces\footnote{Published in \emph{Electron.\ Commun.\ Probab.}, 18 no.\ 17, 2013. MR3037215.}}

\renewcommand{\M}{\Xmm}

\begin{document}

\maketitle
%\symbolfootnote[0]{Email: wolfgang.loehr@uni-due.de}

\begin{abstract}
	The space of metric measure spaces (complete separable metric spaces with a probability measure)
	is becoming more and more important as state space for stochastic processes. Of particular
	interest is the subspace of (continuum) metric measure trees.
	Greven, Pfaffelhuber and Winter introduced the Gromov-Prohorov metric $\dgp$ on the space of metric
	measure spaces and showed that it induces the Gromov-weak topology. They also conjectured that this
	topology coincides with the topology induced by Gromov's $\dg$ metric.
	Here, we show that this is indeed true, and the metrics are even bi-Lipschitz equivalent.
	More precisely, $\dgp=\half\dgh$, and hence $\dgp \les \dg \les 2\dgp$.
	The fact that different approaches lead to equivalent metrics underlines their importance and also that
	of the induced Gromov-weak topology.
	
	As an application, we give a shorter proof of the known fact that the map associating to a lower
	semi-continuous excursion the coded $\R$\nbd tree is Lipschitz continuous when the excursions are
	endowed with the (non-separable) uniform metric. We also introduce a new, weaker, metric topology on
	excursions, which has the advantage of being separable and making the space of bounded excursions a
	Lusin space. We obtain continuity also for this new topology.
\end{abstract}

\section{Introduction}

Tree-valued stochastic processes frequently appear in probability theory and its application areas, such as
theoretical biology. For instance, in an evolutionary model, the development of the genealogical tree is of
interest. In the continuum limit of infinite population size, the finite tree becomes a continuum tree ($\R$\nbd
tree) and the normalised counting measure of individuals becomes a probability measure on it. This measure is
needed to describe the population density on the tree and to sample individuals from it. See Aldous' seminal
paper \cite{Aldous:CRT3} for the convergence of finite variance Galton-Watson trees to a (Brownian) continuum
measure tree, and results of Duquesne and Le~Gall (\cite{DuquesneLeGall:GWtoLevy, Duquesne:limitForContour}) for
the convergence of infinite variance Galton-Watson trees to L\'evy trees.

More generally than $\R$\nbd trees, we can considers random metric (probability) measure spaces, an approach
introduced by Greven, Pfaffelhuber and Winter in \cite{Anita:convRandmmspace} and applied by the authors and
Depperschmidt to obtain tree-valued Fleming-Viot dynamics in \cite{Anita:resampling, DGP:flemingViot}.
Here, $\X\=(X, d, \mu)$ is a \define{metric measure space} (\define{mm-space}) if $(X, d)$ is a
complete, separable metric space and $\mu$ a probability measure on the Borel \sigalg\ of $X$. 
To work with mm-space valued processes, it is crucial to have an appropriate topology on the set of mm-spaces,
or rather the set $\M$ of isometry classes of mm-spaces. A fruitful topology is given by the Gromov-weak
topology introduced in \cite{Anita:convRandmmspace}. In the same paper, the authors conjectured that it coincides
with the topology induced by Gromov's metric $\dg$, which is defined in \cite[Chapter~$3\half$]{Gromov}. They
also introduced a complete metric, the Gromov-Prohorov metric $\dgp$, that metrises the Gromov-weak topology.

Here, we show that $\dg$ and $\dgp$ are bi-Lipschitz equivalent, which in particular implies that the conjecture
is true and $\dg$ indeed metrises Gromov-weak topology. Furthermore, we use this result to prove that the
measure $\R$\nbd tree coded by an excursion depends continuously on the excursion. To this end, we consider two
topologies on the space of lower semi-continuous excursions. For the uniform topology, Lipschitz continuity is
already shown by Abraham, Delmas and Hoscheit in \cite[Prop.~2.9]{AbrahamDelmasHoscheit:exittimes} (with their
metric on trees, which implies the result for ours), but we obtain a much shorter proof using the equivalence of
$\dgp$ and $\dg$.
The uniform topology has the disadvantage of being non-separable, therefore we introduce a new, weaker,
separable, metrisable topology, which is Lusin on the subset of bounded excursions. We also show continuous
dependence of the tree on the excursion in this weaker topology.

In the next section, we recall the definition of the metrics $\dgp$ and $\dg$, as well as of Gromov-weak topology,
and emphasize that the algebra of polynomials used to define Gromov-weak topology is convergence determining
albeit not dense in the bounded continuous functions. We also give a short comparison to related, but
slightly different topologies used on spaces of mm-spaces.
The third section contains the proof of the equivalence of $\dgp$ and $\dg$.
In the last section, we apply the equivalence to measure trees coded by excursions and define the new topology
on the space of excursions.

%	*********************************** DEFINITIONS ***********************************
\section{Metrics and topologies on the space of mm-spaces}

We do not distinguish between isomorphic mm-spaces. Here, two mm-spaces $\X=(X, d, \mu)$ and
$\X'=(X', d', \mu')$ are called \define{isomorphic} if there is a measure preserving map $f\colon X\to X'$ such
that the restriction to the support of $\mu$ is an isometry, i.e.
	\[ \mu' \= \mu\circ f^{-1} \und d(x, y) \= d'\(f(x), f(y)\) \;\; \forall x, y\in\supp(\mu). \]
We denote the space of (isometry classes of) mm-spaces by $\M$. 

\begin{remark}
	Because $(X,d)$ is complete, an isomorphism $f$ from $\X$ to $\X'$ is an isometric bijection between
	$\supp(\mu)$ and $\supp(\mu')$. In particular, there is also an inverse isomorphism $g$ from $\X'$ to
	$\X$ with $g\circ f \= \id$ on $\supp(\mu)$.
\end{remark}

%	.............................. Gromov-Prohorov metric ..............................
\subsubsection*{Gromov-Prohorov metric}

The Gromov-Prohorov metric is obtained by embedding the metric spaces underlying the mm-spaces optimally into a
common metric space and taking the Prohorov distance between the pushforward measures. 

\begin{definition}[Prohorov metric]
	Let $\mu, \nu$ be probability measures on a metric space $(X, d)$. Then the \define{Prohorov distance}
	is
	\[ \dpr(\mu, \nu) \defeq \inf\bset{\eps>0}{\mu(A) \le \nu(A^\eps) + \eps \;\;\forall A\in \B(X)}, \]
	where $A^\eps\defeq\bset{x\in X}{d(A, x) < \eps}$.
\end{definition}

\begin{remark}
	Below, we use the following equivalent expression for the Prohorov metric. A \define{coupling} between
	$\mu$ and $\nu$ is a measure $\xi$ on $X^2=X\times X$ with marginals $\mu$ and $\nu$ on $X$. Then
		\[ \dpr(\mu, \nu) \= \inf\Bset{\eps>0}{\exists \text{ coupling $\xi$ of $\mu$, $\nu$}:
			\xi\(\set{(x,y)\in X^2}{d(x,y)\ge \eps}\) \les \eps}. \]
\end{remark}

\begin{definition}[Gromov-Prohorov metric]
	Let $\X_i=(X_i, d_i, \mu_i) \in \M$, $i=1,2$, be mm-spaces. The \define{Gromov-Prohorov metric} is
	defined by
		\[ \dgp(\X_1, \X_2) \defeq \inf_{f, g}\, \dpr\(\mu_1\circ f^{-1}, \mu_2\circ g^{-1}\), \]
	where the infimum is taken over all isometries $f\colon X_1\to X$ and $g\colon X_2\to X$ into a common
	separable metric space $(X, d)$. 
\end{definition}

%	.............................. Gromov-weak topology ..............................
\subsubsection*{Gromov-weak topology}

The idea of Gromov-weak topology is to use convergence in distribution of finite metric subspaces, which are
sampled from $X$ with the measure $\mu$. A very nice property of the Gromov-Prohorov metric is that it induces
precisely the Gromov-weak topology, as shown in \cite{Anita:convRandmmspace}. This alternative characterisation
of convergence provides us with a sub-algebra of $\Cb(\M)$, called algebra of polynomials. The usefulness of
this algebra stems from the fact that it is rich enough to determine convergence of measures on $\M$.
To emphasize that polynomials are an essential tool for working with convergence in distribution of $\M$\nbd
valued random variables, we remark that one cannot use the space $\Cc(\M)$ of continuous functions with compact
support, because no point in $\M$ has a compact neighbourhood, and hence $\Cc(\M)=\{0\}$ is trivial.

\begin{definition}\deflabel{Gw}
	A \textbf{polynomial} (on $\M$) is a function $\Phi\colon \M\to \R$ of the form
		\[ \Phi(\X) \= \Phi^\phi(\X) \defeq
			\probinta{X^n}{\phi\(\(d(x_i,x_j)\)_{i,j\le n}\)}{\mu^{\otimes n}}{x}, \]
	where $n\iN$ and $\phi\in\Cbdist$. Let $\Pi$ be the set of such functions. \textbf{Gromov-weak
	topology} is the topology induced by $\Pi$ on $\M$.
\end{definition}

\begin{remark}[Polynomials are not dense]
	$\Pi$ is obviously an algebra, but it is not dense in $\Cb(\M)$. To see this, assume it is dense and
	consider the subspace $\M_r$ of mm-spaces with essential diameter bounded by a fixed $r>0$. Because
	$\M_r$ is closed, the set $\Pi_r\defeq\set{\Phi\restricted{\M_r}}{\Phi\in \Pi}$ of restrictions of
	polynomials to $\M_r$ is dense in $\Cb(\M_r)$. Because $\Pi_r$ is clearly separable, this means that
	$\Cb(\M_r)$ is separable, and hence $\M_r$ is compact.
	This is a contradiction (e.g.\ the set of finite spaces with discrete metric and uniform distribution
	has no limit point).
\end{remark}

We say that a set $\F\subseteq \Cb(\M)$ is \define{convergence determining} if for probability measures $\xi_n,
\xi$ on $\M$, the weak convergence $\xi_n \tow \xi$ is equivalent to
	\[ \plainint{f}{\xi_n}\ton\plainint{f}{\xi} \for f\in\F. \]
Since $\Cb(\M)$ is difficult to describe, it is important to have such a set with a more tractable description.
That $\Pi$ is indeed convergence determining is shown with some effort by Depperschmidt, Greven and Pfaffelhuber
in \cite{DGP:mmmspace}. We can also deduce it from an apparently not so well-known general theorem due to Le~Cam.

\begin{theorem}[Le~Cam, \cite{LeCam}; see also {\cite[Lem.~4.1]{HoffmannJorgensen:StFlour}}]
	Let\/ $X$ be a completely regular Hausdorff space, and\/ $\F\subseteq \Cb(X)$ multiplicatively closed.
	Then\/ $\F$ is convergence determining for Radon probability measures if and only if\/ $\F$ generates the
	topology of\/ $X$.
\end{theorem}

\begin{corollary}
	The set\/ $\Pi$ of polynomials is convergence determining.
\end{corollary}
\begin{proof}
	$\M$ is a Polish space, hence completely regular and all probability measures on it are Radon.
	$\Pi$ is an algebra, thus multiplicatively closed and we can apply the Le~Cam theorem.
\end{proof}

%	.............................. Gromov's metric ..............................
\subsubsection*{Gromov's metric $\dgl$}

To obtain the Gromov-Prohorov metric, we embed the metric spaces and measure the distance of the resulting
pushforward measures with the Prohorov metric. For Gromov's $\dgl$ metric, it works the opposite way. Namely,
the measure spaces are parametrised by a measure preserving map from $[0,1]$ (with Lebesgue measure), and then
the distance of the resulting pullbacks of the metrics is evaluated with the following metric.

\begin{definition}[$\dboxl$ metric]
	Let $(X, \B, \mu)$ be a probability space. For functions $r, s\colon X\times X \to \R$, we define
	\[ \dboxl(r, s) \defeq \inf\Bset{\eps>0}{\exists X_\eps\in\B:
		\norm{r\restricted{X_\eps\times X_\eps} -  s\restricted{X_\eps\times X_\eps})}\le\eps,
		\;\; \mu(X\setminus X_\eps)\le \lambda\eps}. \]
\end{definition}

\noindent Obviously, we have
	\[ \dboxl \les \Box_{\lambda'} \les \bruch\lambda{\lambda'} \dboxl  \for \lambda > \lambda'. \]

\begin{definition}[Gromov's $\dgl$ metric]
	Let $\X, \X'$ be mm-spaces, and $I\defeq [0, 1]$, equipped with Lebesgue measure. Let
	$\F(\X)\defeq\bset{\vphi\colon I \to X}{\vphi \text{ is measure preserving}}$ be the set of
	\define{parametrisations} of $(X, \mu)$, and for $\vphi\in\F(\X)$ let $d_\vphi(s, t) \defeq
	d\(\vphi(s), \vphi(t)\)$ be the pullback of $d$ with $\vphi$. Then we define
	\[ \dgl(\X, \X') \defeq
		\inf_{\substack{\vphi\in\F(\X)\\ \vphi'\in\F(\X')}} \dboxl(d_\vphi, d'_{\vphi'}). \]
\end{definition}

\begin{remark}
	Because $(X, d)$ is a Polish space, the set $\F(\X)$ of (measure preserving) parametrisations is
	non-empty. This follows for example from the version of the Skorohod representation on $I$ given in
	\cite[Thm.~8.5.4]{BogachevII}.
\end{remark}

%	.............................. Other topologies ..............................
\subsubsection*{Related topologies}

\begin{enumerate}
	\item In \cite{measuredGH}, Fukaya introduced the \define{measured Hausdorff topology} (often cited as
		measured Gromov-Hausdorff topology) for compact mm-spaces. 
		The same topology is called \define{weighted Gromov-Hausdorff topology},
		and a complete metric inducing it is constructed by Evans and Winter in
		\cite{Anita:subtreeprune}. The idea is that spaces are close if there is an $\eps$\nbd
		isometry mapping one measure Prohorov-close to the other.
		Convergence in measured Hausdorff topology implies Gromov-weak convergence, but not vice
		versa, because the former implies Gromov-Hausdorff convergence of the underlying
		metric spaces, which is not the case for Gromov-weak topology. Note that the underlying
		equivalence classes are also different: For two mm-spaces to be equivalent in the measured
		Hausdorff topology, the whole spaces have to be isometric, while in a Gromov-weak sense, this
		is required only for the supports of the measures.
%		only the supports of the measures have to be isometric.
	\item Recently, Abraham, Delmas and Hoscheit (\cite{AbrahamDelmasHoscheit:dGHP}) extended the measured
		Hausdorff topology to complete, locally compact, rooted length spaces with locally finite
		measures.  Note that these measures are finite on all balls, because closed balls are compact in
		such spaces.
		The authors introduced the \define{Gromov-Hausdorff-Prohorov metric}, first on compact spaces
		using an embedding and measuring the sum of Hausdorff and Prohorov distance. That this metrises
		measured Hausdorff topology is easy to see from the definitions, using the same connection
		between $\eps$\nbd isometries and Hausdorff-close embeddings that is frequently applied in the
		context of Gromov-Hausdorff convergence.
		In the locally compact setting, they integrate the weighted distances of the measures restricted
		to balls.  Note that this extended topology is vague in the sense that the total mass is not
		preserved. Thus, on spaces with finite (not necessarily probability) measures, it is not
		stronger than the natural extension of Gromov-weak topology, where the measures in \defref{Gw}
		are no longer required to be probabilities.
	\item In \cite{Sturm:geometryMMspace}, Sturm defines the \define{$L_2$\nbd transportation distance}
		analogously to $\dgp$, but with the (2-)Wasserstein metric instead of the Prohorov metric. It
		induces a topology on $\M$ that is strictly stronger than Gromov-weak topology, but coincides
		with it on subspaces of $\M$ consisting of spaces with uniformly bounded (essential) diameter.
		Its restriction to the space of compact mm-spaces is strictly weaker than measured Hausdorff
		topology.
\end{enumerate}

%	*********************************** EQUIVALENCE of the METRICS ***********************************
\section{Equivalence of $\dgp$ and $\dg$}

\begin{theorem}\thmlabel{main}
	$\texteq \dgp \= \half\dgh$.
\end{theorem}
\begin{proof} Let $\X_i=(X_i, d_i, \mu_i)$, $i=1,2$, be mm-spaces.
	\wcase{``$\ge$''} Assume $\dgp(\X_1, \X_2) \ls \eps$ for some $\eps>0$. Then we can embed $(X_i, d_i)$,
		$i=1,2$, into a (common) complete, separable metric space $(X, d)$, such that the pushforward
		measures $\nu_i$ satisfy $\dpr(\nu_1, \nu_2) \ls \eps$. Thus there is a coupling $\nu$ of
		$\nu_1$ and $\nu_2$ on $X^2$ with
			\[ \nu(Y_\eps) \les \eps \spacedtxt{for} Y_\eps
				\defeq \bset{(x, y) \in X^2}{d(x, y) \ges \eps}.\]
		Now choose a parametrisation $\vphi$ of $(X^2, \nu)$, i.e.\ $\vphi\colon[0,1]\to X^2$ is
		measurable and $\nu=\lambda\circ\vphi^{-1}$ for Lebesgue measure $\lambda$.
		Let $\pi_i$, $i=1,2$, be the canonical projections from $X^2$ to $X$. Then $\vphi_i\defeq
		\pi_i\circ\vphi$ is a parametrisation of $\X_i$ (or its isomorphic image in $X$).
		Let $r_i$ be the pullback of $d$ under $\vphi_i$. We show $\dboxh(r_1, r_2) \les 2\eps$.
		Indeed, $\lambda\(\vphi^{-1}(Y_\eps)\) \= \nu(Y_\eps)\les \eps \= \half 2\eps$, and for
		$s, t\in[0,1]\setminus \vphi^{-1}(Y_\eps)$ we have by definition of $Y_\eps$ that
		$d\(\vphi_1(s), \vphi_2(s)\) \les \eps$. Thus
			\[ r_1(s, t) \= d\(\vphi_1(s), \vphi_1(t)\) \les d\(\vphi_2(s), \vphi_2(t)\) + 2\eps
				\= r_2(s, t)+2\eps, \]
		and by symmetry, $\betrag{r_1(s, t) - r_2(s, t)}\les 2\eps$.
		In total, $\dgh(\X_1, \X_2) \les \dboxh(r_1, r_2) \les 2\eps$.
	\wcase{``$\le$''} Let $\dgh(\X_1, \X_2)\ls 2\eps$ and $\vphi_i \colon [0,1] \to X_i$ parametrisations
		of $\X_i$, $i=1,2$, with $\dboxh(r_1, r_2) \ls 2\eps$, where $r_i$ is the pullback of $d_i$
		with $\vphi_i$. There is a set $S\subseteq [0,1]$ with $\lambda(S)\ges 1-\eps$ and
		$|r_1-r_2|\les 2\eps$ on $S^2$. On the disjoint union $X\defeq X_1\uplus X_2$, we define a
		metric $d$ by
		\begin{equation}\eqlabel{ddef}
			d\restricted{X_i^2}\defeq d_i \und d(x, y) \defeq
			     \inf_{s\in S} d_1\(x,\vphi_1(s)\) + d_2\(\vphi_2(s), y\) + \eps
			     	\;\;\forall x\in X_1,\,y\in X_2.
		\end{equation}
		We check that $d$ satisfies the \triin\ in \lemref{metric} below. Extend the $\mu_i$ to measures
		on $X$ with support in $X_i$. To estimate their Prohorov distance in $(X, d)$, let
		$F\subseteq X$ be measurable. Note that by definition, $d\(\vphi_1(s), \vphi_2(s)\) \= \eps$ for
		every $s\in S$. Consequently, for every $\eps_0 \gs \eps$,
			\[ \vphi_2\(\vphi_1^{-1}(F) \cap S\) \subseteq F^{\eps_0} \spacedtxt{where}
				F^{\eps_0} \= \bset{x\in X}{d(x, F) \ls \eps_0}. \]
		Therefore,
			\[ \mu_1(F) \= \lambda\(\vphi_1^{-1}(F)\) \les \lambda\(\vphi_1^{-1}(F)\cap S\) + \eps
				\les \mu_2 \( \vphi_2\(\vphi_1^{-1}(F) \cap S\)\) + \eps
				\les \mu_2(F^{\eps_0}) + \eps.\]
		Since $\eps_0\gs \eps$ is arbitrary, $\dpr(\mu_1, \mu_2)\les \eps$ and thus
		$\dgp(\X_1, \X_2)\les \eps$.
\end{proof}	% end proof of main theorem

\begin{corollary}\corlabel{equiv}
	For every\/ $\lambda>0$, we have
	\[ \min\{2,\tfrac1\lambda\}\cdot\dgp \les \dgl \les \max\{2,\tfrac1\lambda\}\cdot \dgp. \]
	In particular, $\dg$ induces the Gromov-weak topology.
\end{corollary}
\begin{proof}
	For $\lambda\ge \half$, the equation $\dgh\les 2\lambda\dgl \les 2\lambda\dgh$ is obvious from the
	definition of $\dg[\lambda]$. For $\lambda \le \half$, we get the same inequality with
	``$\ge$'' instead of ``$\le$''. Now the theorem implies the claim.
\end{proof}

\noindent We still have to check that \eqref{ddef} in the proof of \thmref{main} defines a metric.

\begin{lemma}\lemlabel{metric}
	The\/ $d$ defined in \eqref{ddef} satisfies the\/ \triin. Thus it is a metric. 
\end{lemma}
\begin{proof}
	For $x, m\in X_1, \, y \in X_2$, we have
		\[ d(x, y) \les \inf_{s\in S} d_1(x, m) + d_1\(m, \vphi_1(s)\) + d_2\(\vphi_2(s), y\) + \eps
	  		\= d(x, m) + d(m, y). \]
	For $x, y\in X_1,\, m\in X_2$, we have
	\begin{eqnarray*}
		d(x, y) &\le& \inf_{s,t\in S} d_1\(x, \vphi_1(s)\) + d_1\(\vphi_1(s), \vphi_1(t)\)
					+ d_1\(\vphi_1(t), y\)\\
			&\le& \inf_{s,t\in S} d_1\(x, \vphi_1(s)\) + d_2\(\vphi_2(s), \vphi_2(t)\)
					+ d_1\(\vphi_1(t), y\) + 2\eps \\
			&\le& \inf_s d_1\(x, \vphi_1(s)\) + d_2\(\vphi_2(s), m\) + \eps
					+ \inf_t  d_2\(m, \vphi_2(t)\) + d_1\(\vphi_1(t), y\) + \eps \\
			&=& d(x, m)+d(m, y).
	\end{eqnarray*}
	All other cases follow by symmetry or by the \triins\ in $X_1$ and $X_2$.
\end{proof}

%	*********************************** CODING with EXCURSIONS ***********************************
\section{Continuity of the coding of $\R$-trees by excursions}

An $\R$-tree (see \cite{Dress:T-theory}) is a complete, connected 0-hyperbolic metric space $(T, d)$.
One of the possible definitions of $0$\nbd hyperbolicity is that it satisfies the four point condition, i.e.
	\[ d(v_1, v_2) + d(v_3, v_4) \les \max\bsset{d(v_1, v_3) + d(v_2, v_4),\, d(v_1, v_4) + d(v_2, v_3)}
	\for v_1,\ldots,v_4\in T. \]
Note that every $0$\nbd hyperbolic space can be embedded isometrically into a unique smallest $\R$\nbd tree (see
\cite[Thm.~3.38]{Evans:StFlour}), which is separable whenever the original space was separable. Because $\dgp$ (unlike
the measured Hausdorff topology) identifies a metric measure space with every subspace containing the support of
the measure, the equivalence class of every $0$\nbd hyperbolic space contains an $\R$\nbd tree. 

One possibility to construct $0$\nbd hyperbolic spaces is to code them by excursions, see \cite{Aldous:CRT3,
LeGall:treeInBex, DuquesneLeGall:GWtoLevy}. To this end, let $h\colon [0,1] \to \R_+$ be a positive function with
$h(0)=0$, and consider the semi-metric
	\[ d_h(s,t) \defeq h(s) + h(t) - 2I_h(s,t),
		\wspace I_h(s,t) \defeq \inf_{u\in [s\land t,\, s\lor t]} h(u), \]
on $[0,1]$. Then the quotient space $T_h \defeq [0,1]\quotient{d_h}$ is a $0$\nbd hyperbolic metric space.
We additionally assume that $h$ is lower semi-continuous. Then $T_h$ is separable and the natural projection
	\[ \pi_h \colon [0,1] \to T_h \]
is measurable. To see this, note that the canonical projection from the graph
$\gr(h)=\bset{(t, h(t))}{t\in[0,1]}\subseteq \R^2$ of $h$ onto the tree $T_h$ is continuous due to lower
semi-continuity of $h$.
$T_h$ needs to be neither complete nor connected, but we identify it with its completion and, once we have put a
measure on it, the equivalence class contains a connected representative.

\begin{enremark}
	\item If the graph of $h$ is connected, then $T_h$ is complete and connected to begin with. We do not,
		however, make this restriction.
	\item If $h$ is continuous, $\pi_h$ is continuous and $T_h$ is compact.
		Conversely, every compact $\R$\nbd tree can be coded by a (non-unique) continuous excursion
		(\cite[Rem.~3.2]{Anita:subtreeprune}). To code compact \emph{measured} trees, continuous
		excursions are not sufficient. 
		See \cite{Duquesne:codingCompact} for a detailed account on coding compact, rooted,
		ordered, measured $\R$\nbd trees in a unique way by upper semi-continuous c\`agl\`ad
		excursions.
\end{enremark}

\begin{definition}
	We define the set of (generalised) \define{excursions} on $[0,1]$ as
		\[ \ex\defeq \bset{h\colon [0,1]\to\R_+}{h(0)=0,\;\text{$h$ lower semi-continuous}}. \]
	Let $\exb$ be the subset of bounded functions in $\ex$.
	For $h\in \ex$, let the mass measure $\mu_h$ on $T_h$ be the image of Lebesgue measure $\lambda$ under
	$\pi_h$ and define the \textbf{coding function}
		\[ \code\colon \ex \to \M, \qquad h \mapstos \T_h \defeq (T_h, d_h, \mu_h).\]
\end{definition}

It is shown in \cite[Prop.~2.9]{AbrahamDelmasHoscheit:exittimes} that the coding function
$\code$ is Lipschitz continuous when the space of excursions is equipped with the uniform metric and the space
of trees with the Gromov-Hausdorff-Prohorov metric. For the Gromov-Prohorov metric, this is a slightly weaker
statement. The proof, however, becomes trivial in this case if we use \thmref{main}, because the trees are
already given in a parameterised form.

\begin{proposition}
	Let\/ $h, g \in \ex$. Then
		\[ \dgp(\T_h, \T_g) \les 2 \norm{h-g} \= 2\sup_{t\in[0,1]} \betrag{h(t)-g(t)}. \]
\end{proposition}
\begin{proof}
	$\ds \dgp(\T_h, \T_g) \= \half \dgh(\T_h, \T_g) \les \half \Box_\half(d_h, d_g) \les 2\norm{h-g}.$
\end{proof}

The uniform metric on $\ex$ is a rather strong one, in particular $\ex$ and $\exb$ are not separable in this
metric. The coding function turns out to be still continuous if we equip $\ex$ with a weaker, separable,
metrisable topology, namely the weakest topology which is stronger than convergence in measure and epigraph
convergence.  For $h,h'\in\ex$, let
	\[ \dm(h, h') \defeq \inf\Bset{\eps>0}{\lambda\(\bset{t}{|h(t)-h'(t)|>\eps}\) < \eps}, \]
which metrises convergence in Lebesgue measure, $d_H$ the Hausdorff metric in $\R^2$, and
	\[ \dhepi(h, h') \defeq d_H\(\epi(h), \epi(h')\), \qquad \epi(h) \defeq \bset{(t,y)\in [0,1]\times \R_+}{y\ge h(t)}.\]
Note that the epigraph of a function is closed if and only if the function is lower semi-continuous.
Epigraph convergence is usually defined as convergence in Fell topology (or equivalently Kuratowski convergence)
of the epigraphs, see e.g.\ \cite{Beer:topOnSets}. It is a compact, metrisable topology on the set $\excl$ of
$(\R_+\cup\{\infty\})$\nbd valued, lower semi-continuous functions on $[0,1]$. On $\excl$, the topology induced
by $\dhepi$ is strictly stronger. Restricted to $\ex$, however, the topologies coincide, which follows from
\cite[Thm.~1]{Beer:epiconv} using compactness of $[0,1]$ and $\R$-valuedness of excursions.
Epigraph convergence also coincides with $\Gamma$\nbd convergence (see e.g.\ \cite{DalMaso:gammaconv}),
whence the name $\dhepi$.

\begin{definition}
	We endow\/ $\ex$ with the \textbf{excursion metric} $ \dex \defeq \dhepi + \dm $.
\end{definition}

Recall that a metrisable topological space $X$ is called \emph{Lusin space} if it is the continuous, injective
image of a Polish space, i.e.\ if there exists a Polish space $Y$ and a continuous bijection $f\colon Y \to X$.
$X$ is Lusin if and only if it is homeomorphic to a Borel subset of a Polish space (see \cite[Sec.~8.6]{Cohn80}
for details).

\begin{proposition}
	$\ex$ is a separable metric space, and the set of continuous excursions is dense.
	Furthermore, $\exb$ is a Lusin space.
\end{proposition}
\begin{proof}
	$\dex$ is obviously a metric, and the continuous excursions are both $\dhepi$\nbd dense (increasing
	pointwise convergence implies $\dhepi$\nbd convergence) and $\dm$\nbd dense in $\ex$. Hence $\ex$ is
	separable, and it remains to show that $\exb$ is a Borel subset of a Polish space.
	First note that this is the case for $(\exb, \dhepi)$, because the set of excursions bounded by a fixed
	$M\in \N$ is closed in the compact metric space $\excl$ with epigraph topology.
	Now we can identify $(\exb, \dex)$ with the graph of the function $\pi\colon (\exb, \dhepi)
	\to L^0\defeq \(L^0(\lambda), \dm\)$, which maps an excursion to its $\lambda$-a.e.\ equivalence class.
	It is enough to show that $\pi$ is measurable, because then $(\exb, \dex) \simeqq \gr(\pi)$ is an
	injective measurable image of a Lusin space, hence Lusin itself by \cite[Thm.~8.3.7]{Cohn80}.

	To show measurability, choose a fixed dense sequence $\folge{f}$ of continuous excursions, and define
	$\pi_n\colon \exb \to L^0,\; h \mapsto \sup_{f_k\le h,\,k\le n} f_k$. 
	Then $\pi_n$ is a simple function and measurable, because $\set{h\in \exb}{h\ge f_k}$ is closed in
	$(\exb, \dhepi)$. Because $h=\sup_{f_n \le h} f_n$, $\pi$ is the pointwise limit of the $\pi_n$, thus
	also measurable.
\end{proof}

\begin{example}[$\dex$ is not complete and $\code$ is not uniformly continuous]
	Let $h_n(t) = 1-\unit_{\N_0}(nt)$, $t\in[0,1]$.
	Then $h_n$ codes the discrete space of $n$ points with uniform distribution or, equivalently, the
	star-shaped tree with $n$ leaves and uniform distribution on the leaves. $h_n$ converges in epigraph
	topology to the zero function, while $\dm(h_n, \unit)=0$ for each $n$.
	Thus $\folge{h}$ is Cauchy w.r.t.\ $\dex$, but does not converge. $\(\code(h_n)\)_{n\iN}$
	is not a Cauchy sequence in $\M$, hence $\code$ is not uniformly continuous.
\end{example}

\begin{remark}
	We do not know if $\ex$ is Lusin or even Polish.
	$\exb$ is not Polish, because it is a dense \Fsset\ (countable union of closed sets) with dense
	complement (in $\ex$).

	That such a set cannot be Polish can be seen as follows. Let $A_n$ be closed with dense complement in
	$\ex$. Then its closure $\closure{A}_n$ in $\excl$ is closed with empty interior in the Polish space $\excl$.
	Assume that $A:=\bigcup_{n\iN} A_n$ is Polish. By the Mazurkiewicz theorem (\cite[Thm.~8.1.4]{Cohn80}),
	$A$ is a \Gdset\ in $\excl$, i.e.\ $A=\bigcap_{n\iN} U_n$ for some open sets $U_n\subseteq \excl$. Let
	$A'_n := \excl \setminus U_n$.
	Then $\excl=\bigcup_{n\iN} \(\closure{A}_n \cup A'_n\)$ and by the Baire category theorem
	(\cite[Thm.~D.37]{Cohn80}), at least one $A'_n$ has to have non-empty interior.
	This means that $A$ is not dense.
\end{remark}

\begin{theorem}
	The coding function\/ $\code\colon \ex\to \M$ is continuous (w.r.t.\ $\dex$ and\/ $\dgp$).
\end{theorem}
\begin{proof}
	Fix $h\in\ex$, $\eps>0$. We construct a $\delta>0$ such that
	$\Box_{\mspace{-1mu}1}(d_h, d_g) \les 6\eps$ for every $g\in\ex$ with $\dex(h, g)\les \delta$. Then
	\corref{equiv} implies the result.
\begin{enumerate}
	\item Let $A_\eta \defeq \bset{t\in [0,1]}{I_h(t-\eta, t+\eta) < h(t) - \eps}$. Because $h$ is lower
		semi-continuous, $A_\eta \searrow \emptyset$ for $\eta\to 0$. Thus there is a $0<\delta<\eps$ with
		$\lambda(A_\delta) < \eps$. Fix $g\in\ex$ with $\dex(h, g)\les\delta$ and let $X_\eps \defeq
		[0,1] \setminus \(A_\delta \cup \sset{|h-g|>\delta}\)$. Then $\lambda\([0,1]\setminus X_\eps\)
		\les 2\eps$ and it is enough to show $\betrag{d_h(s, t) - d_g(s,t)} \les 6\eps$ for $s,t\in X_\eps$.
		Because $h$ and $g$ are $\eps$-close at $s$ and $t$, this is satisfied once we have shown
		$\betrag{I_h(s,t) - I_g(s, t)} \les 2\eps$.
	\item ``$I_g\le I_h+2\eps$'':
		Because $h$ is lower semi-continuous, the infimum $I_h(s,t)$ is attained and there is a
		$u\in [s,t]$ with $h(u) = I_h(s,t)$. From $\dhepi(h,g)\les \delta$, we obtain the existence of
		$u'\in [u-\delta, u+\delta]$ with $g(u') \les h(u) + \delta$. If $u' \in [s,t]$, then
		$I_g(s,t)\le g(u') \le h(u)+\delta \le I_h(s,t)+\eps$. For the case $u'\not\in[s,t]$,
		assume w.l.o.g.\ $u'<s$, and therefore $u\in [s, s+\delta]$. Then, because $s$ is not in
		$A_\delta$, we have
		$I_h(s,t) \= h(u) \ge {h(s) - \eps} \ge {g(s) - 2\eps} \ge I_g(s,t) - 2\eps$.
	\item ``$I_h\le I_g+2\eps$'':
		Choose $u\in [s,t]$ with $g(u) = I_g(s,t)$ and $u'\in [u-\delta,u+\delta]$ with
		$h(u') \le {g(u) + \delta}$.
		As above we can assume $u \in [s, s+\delta]$, $u' \in [s-\delta, s]$ and obtain
		$I_h(s,t) \le h(s) \le {h(u') + \eps} \le g(u) + 2\eps = I_g(s, t) + 2\eps$.
\qedhere\end{enumerate}
\end{proof}

\begin{acknowledgements}
	I am thankful to Anita Winter for discussions, encouragement, and helpful comments on the previous
	version of the manuscript. I also thank Guillaume Voisin for many discussions about trees, Patrick
	Hoscheit for a discussion about topologies on the space of excursions, and the referees for
	helpful comments.
\end{acknowledgements}

\smaller
\bibliography{bib-math,mm-space}

\end{document}

%% file: macros-gromov.tex
\makeatletter
\renewcommand\section{\@startsection {section}{1}{\z@}%
                                   {-2.5ex \@plus -1ex \@minus -.2ex}%
                                   {1.1ex \@plus.2ex}%
                                   {\normalfont\large\bfseries}}
\makeatother

\newcommand{\ex}{\mathcal{E}}
\newcommand{\excl}{\closure{\ex}}
\newcommand{\exb}{\ex_b}
\newcommand{\dex}{d_\ex}
\newcommand{\dm}{d_\lambda}
\newcommand{\dhepi}{d_\Gamma}
\newcommand{\code}{\mathfrak{C}}
\newcommand{\epi}{\mathrm{epi}}
\newcommand{\gr}{\mathrm{gr}}
\newcommand{\Cbdist}{\Cb\(\R^{n\times n}\)}
\renewcommand{\tow}{\tonamed{w}}

\newcommand{\F}{\mathcal{F}}

\newcommand{\B}{\mathfrak{B}}

\newcommand{\dboxl}{\Box_\lambda}
\newcommand{\dboxh}{\Box_\half}
\newcommand{\dgp}{d_\mathrm{GP}}
\newcommand{\dg}[1][1]{\underline\Box_{#1}}
\newcommand{\dgl}{\dg[\lambda]}
\newcommand{\dgh}{\dg[\half]}
\newcommand{\dpr}{d_\mathrm{Pr}}

\newcommand{\M}{\protect\mathds{M}}
\newcommand{\X}{\protect\mathcal{X}}
\newcommand{\T}{\mathcal{T}}
\newcommand{\Xmm}{\protect\mathfrak{X}}

\newcommand{\triin}{$\triangle$\nbd inequality}
\newcommand{\triins}{$\triangle$\nbd inequalities}

%% file: gromov-metrik.bbl
\begin{thebibliography}{GPW13}

\bibitem[ADH13]{AbrahamDelmasHoscheit:dGHP}
Romain Abraham, Jean-Fran\c{c}ois Delmas, and Patrick Hoscheit.
\newblock A note on the {G}romov-{H}ausdorff-{P}rokhorov distance between
  (locally) compact metric measure spaces.
\newblock {\em Electron. J. Probab.}, 18(14):1--21, 2013.

\bibitem[ADH14]{AbrahamDelmasHoscheit:exittimes}
Romain Abraham, Jean-Fran{\c{c}}ois Delmas, and Patrick Hoscheit.
\newblock Exit times for an increasing {L\'e}vy tree-valued process.
\newblock {\em Probab.\ Theo.\ Rel.\ Fields}, 159(1-2):357--403, 2014.

\bibitem[Ald93]{Aldous:CRT3}
David Aldous.
\newblock The continuum random tree {III}.
\newblock {\em Annals of Prob.}, 21(1):248--289, 1993.

\bibitem[Bee93]{Beer:topOnSets}
Gerald Beer.
\newblock {\em Topologies on Closed and Closed Convex Sets}.
\newblock Kluwer Acad. Publ., 1993.

\bibitem[Bee94]{Beer:epiconv}
Gerald Beer.
\newblock A note on epi-convergence.
\newblock {\em Canad. Math. Bull.}, 37(3):294--239, 1994.

\bibitem[Bog07]{BogachevII}
V.~I. Bogachev.
\newblock {\em Measure Theory, Volume {II}}.
\newblock Springer, 2007.

\bibitem[Coh80]{Cohn80}
Donald~L. Cohn.
\newblock {\em Measure Theory}.
\newblock Birkh{\"a}user, 1980.

\bibitem[DGP11]{DGP:mmmspace}
Andrej Depperschmidt, Andreas Greven, and Peter Pfaffelhuber.
\newblock Marked metric measure spaces.
\newblock {\em Electron. Commun. Prob.}, 16:174--188, 2011.

\bibitem[DGP12]{DGP:flemingViot}
Andrej Depperschmidt, Andreas Greven, and Peter Pfaffelhuber.
\newblock Tree-valued {Fleming-Viot} dynamics with mutation and selection.
\newblock {\em Annals of Applied Prob.}, 22(6):2560--2615, 2012.

\bibitem[DLG02]{DuquesneLeGall:GWtoLevy}
Thomas Duquesne and Jean-Fran{\c{c}}ois Le~Gall.
\newblock Random trees, {L\'evy} processes and spatial branching processes.
\newblock {\em Ast{\'e}risque}, 281:vi+147, 2002.

\bibitem[DMT96]{Dress:T-theory}
Andreas~W.M. Dress, V.~Moulton, and W.F. Terhalle.
\newblock T-theory: An overview.
\newblock {\em Europ. J. Combinatorics}, 17(2-3):161--175, 1996.

\bibitem[Duq03]{Duquesne:limitForContour}
Thomas Duquesne.
\newblock A limit theorem for the contour process of conditioned
  {Galton-Watson} trees.
\newblock {\em Annals of Prob.}, 31(2):996--1027, 2003.

\bibitem[Duq06]{Duquesne:codingCompact}
Thomas Duquesne.
\newblock The coding of compact real trees by real valued functions, 2006.
\newblock arXiv:0604106.

\bibitem[Eva07]{Evans:StFlour}
Steven~N. Evans.
\newblock Probability and real trees.
\newblock In {\em {\'Ecole d'\'Et\'e} de {Probabilit\'es} de Saint Flour
  XXXV-2005}, volume 1920 of {\em Lecture Notes in Mathematics}, pages 1--193.
  Springer, 2007.

\bibitem[EW06]{Anita:subtreeprune}
Steven~N. Evans and Anita Winter.
\newblock Subtree prune and regraft: a reversible real tree-valued markov
  process.
\newblock {\em Annals of Prob.}, 34(3):918--961, 2006.

\bibitem[Fuk87]{measuredGH}
Kenji Fukaya.
\newblock Collapsing of {Riemannian} manifolds and eigenvalues of {Laplace}
  operator.
\newblock {\em Inventiones Math.}, 87(3):517--547, 1987.

\bibitem[GPW09]{Anita:convRandmmspace}
Andreas Greven, Peter Pfaffelhuber, and Anita Winter.
\newblock Convergence in distribution of random metric measure spaces
  ({$\Lambda$}-coalescent measure trees).
\newblock {\em Prob. Theo. Rel. Fields}, 145(1-2):285--322, 2009.

\bibitem[GPW13]{Anita:resampling}
Andreas Greven, Peter Pfaffelhuber, and Anita Winter.
\newblock Tree-valued resampling dynamics. {M}artingale problems and
  applications.
\newblock {\em Prob. Theo. Rel. Fields}, 155:789--838, 2013.

\bibitem[Gro99]{Gromov}
Misha Gromov.
\newblock {\em Metric Structures for Riemannian and Non-Riemannian Spaces}.
\newblock Birkh{\"a}user, 1999.

\bibitem[HJ77]{HoffmannJorgensen:StFlour}
J.~Hoffmann-J{\o}rgensen.
\newblock Probability in {Banach} spaces.
\newblock In {\em {{\'Ecole d'\'Et\'e} de {Probabilit\'es} de Saint Flour
  VI-1976}}, volume 598 of {\em Lecture Notes in Mathematics}. Springer, 1977.

\bibitem[LC57]{LeCam}
L.~Le~Cam.
\newblock Convergence in distribution of stochastic processes.
\newblock {\em University of California Publications in Statistics},
  2:207--236, 1957.

\bibitem[LG93]{LeGall:treeInBex}
Jean-Fran{\c{c}}ois Le~Gall.
\newblock The uniform random tree in a {Brownian} excursion.
\newblock {\em Prob. Theo. Rel. Fields}, 96(3):369--383, 1993.

\bibitem[Mas93]{DalMaso:gammaconv}
Gianni~Dal Maso.
\newblock {\em An Introduction to {$\Gamma$}-Convergence}.
\newblock Birkh{\"a}user, 1993.

\bibitem[Stu06]{Sturm:geometryMMspace}
Karl-Theodor Sturm.
\newblock On the geometry of metric measure spaces {I}.
\newblock {\em Acta Math.}, 196(1):65--131, 2006.

\end{thebibliography}
